\documentclass[12pt,a4paper,english,reqno]{amsart}
\usepackage[a4paper,footskip=1.5em]{geometry}
\usepackage{amsmath,amssymb,amsthm,mathtools,comment}
\usepackage[mathscr]{euscript}
\usepackage{tikz,babel,enumitem,adjustbox}
\usepackage[final]{microtype}
\usepackage[numbers]{natbib}
\usetikzlibrary{decorations.pathreplacing,calc}
\usepackage{hyperref}
\hypersetup{colorlinks=true,linkcolor=blue,citecolor=blue,pdfpagemode=UseNone,pdfstartview={XYZ null null 1.00}}
\usepackage{cmtiup}

\theoremstyle{plain}
\newtheorem{theorem}{Theorem}[section]		
\newtheorem{lemma}[theorem]{Lemma}
\newtheorem{claim}[theorem]{Claim}
\newtheorem{proposition}[theorem]{Proposition}

\theoremstyle{remark}

\def\CC{\mathcal{C}}
\def\TT{\mathcal{D}}
\def\TTT{\mathcal{T}}
\def\RR{\mathcal{R}}
\def\BB{\mathcal{B}}
\def\DD{\mathscr{U}}
\def\N{\mathbb{N}}

\DeclareMathOperator\Deg{d}
\DeclareMathOperator\ex{ex}
\newcommand{\eps}{\ensuremath{\varepsilon}}

\let\originalleft\left
\let\originalright\right
\renewcommand{\left}{\mathopen{}\mathclose\bgroup\originalleft}
\renewcommand{\right}{\aftergroup\egroup\originalright}

\makeatletter
\def\imod#1{\allowbreak\mkern10mu({\operator@font mod}\,\,#1)}
\makeatother

\begin{document}
\title{Tur\'an theorems for unavoidable patterns}

\author{Ant\'onio Gir\~ao}
\address{School of Mathematics, University of Birmingham, Edgbaston, Birmingham, B15\thinspace2TT, UK}
\email{giraoa@bham.ac.uk}

\author{Bhargav Narayanan}
\address{Department of Mathematics, Rutgers University, Piscataway, NJ 08854, USA}
\email{narayanan@math.rutgers.edu}

\date{8 June 2019}

\begin{abstract}
	We prove Tur\'an-type theorems for two related Ramsey problems raised by Bollob\'as and by Fox and Sudakov. First, for $t \ge 3$, we show that any two-colouring of the complete graph on $n$ vertices that is $\delta$-far from being monochromatic contains an \emph{unavoidable $t$-colouring} when $\delta \gg n^{-1/t}$, where an unavoidable $t$-colouring is any two-colouring of a clique of order $2t$ in which one colour forms either a clique of order $t$ or two disjoint cliques of order $t$. Next, for $ t\ge 3$, we show that any tournament on $n$ vertices that is $\delta$-far from being transitive contains an \emph{unavoidable $t$-tournament} when $\delta \gg n^{-1/\lceil t/2 \rceil}$, where an unavoidable $t$-tournament is the blow-up of a cyclic triangle obtained by replacing each vertex of the triangle by a transitive tournament of order $t$. Conditional on a well-known conjecture about bipartite Tur\'an numbers, both results are sharp up to implied constants and hence determine the order of magnitude of the corresponding off-diagonal Ramsey numbers.
\end{abstract}

\maketitle

\section{Introduction}
The starting point of Ramsey theory, namely Ramsey's theorem~\citep{fpr}, is the assertion that given any natural number $t\in \N$, every two-colouring (of the edges, here and elsewhere) of the complete graph $K_n$ on $n$ vertices contains a monochromatic copy of $K_t$ for all large enough $n\in \N$; the asymptotic behaviour of the smallest such integer, namely the Ramsey number $\RR(t)$, has been the subject of intense scrutiny (see~\citep{lowerramsey, upperramsey, Thomason1988, Conlon2009}, for example) through the past seventy or so years.

A priori, one cannot expect to find any non-monochromatic patterns in a given two-colouring of a complete graph, since the colouring in question might itself be monochromatic. In the light of this, Bollob\'as~\citep{cb} asked what non-monochromatic patterns one is guaranteed to find in any two-colouring of $K_n$ that is $\delta$-far from being monochromatic for some $\delta > 0$, where a two-colouring of $K_n$ is said to be \emph{$\delta$-far from being monochromatic} if each colour in the colouring, henceforth red and blue, appears on at least $\delta n^2$ edges. Call a two-colouring of $K_{2t}$ an \emph{unavoidable $t$-colouring} if one colour class forms either a clique of order $t$ or two disjoint cliques of order $t$. It is not hard to see that the only non-monochromatic patterns one could hope to find, even in a given two-colouring of $K_n$ that is far from being monochromatic, are precisely such unavoidable colourings (since the given colouring might itself be of this form). Confirming Bollob\'as's prediction, Cutler and Montagh showed that for any $t \in \N$ and $\delta > 0$, there exists a least integer $\CC(t,\delta) \in \N$ such that any two-colouring of $K_n$ that is $\delta$-far from being monochromatic contains an unavoidable $t$-colouring for all $n \ge \CC(t,\delta)$. Turning to quantitative estimates, Fox and Sudakov~\citep{fs} subsequently determined the order of growth of the Ramsey number $\CC(t,\delta)$, showing that $\CC(t,\delta) = (1/\delta)^{\Theta(t)}$.

Our first result pins down the off-diagonal growth rate of the Ramsey number $\CC(t,\delta)$ for each $t \ge 3$ as $\delta \to 0$. Our first result in its Tur\'an-type formulation is as follows. 
\begin{theorem}\label{colourthm}
	For each integer $t \ge 3$, there exists a $C = C(t)>0$ such that any two-colouring of the complete graph on $n \ge C$ vertices that is $Cn^{-1/t}$-far from being monochromatic contains an unavoidable $t$-colouring.
\end{theorem}
A well-known conjecture of Kov\'ari, S\'os and Tur\'an~\citep{kst}, see also~\citep{survey}, asserts that
\begin{equation}\label{kstconj}
	\ex(n, K_{a,b}) = \Omega(n^{2-1/a})
\end{equation} for all $b\ge a \ge 2$. Conditional on the truth of~\eqref{kstconj} in the case of $a = b = t$, our first result is easily seen to be sharp up to the multiplicative factor $C(t)$ in its statement for each $t \ge 3$. Theorem~\ref{colourthm} therefore determines, conditional on~\eqref{kstconj}, the order of magnitude of $\CC(t,\delta)$ when $t$ is much smaller than $1/\delta$; indeed, it follows from our result that, for each fixed $t \ge 3$ and in the limit of $\delta \to 0$, we have
\begin{equation}\label{colbound}
	\CC(t,\delta) = \Theta\left(\left(\frac{1}{\delta}\right)^{t}\right).
\end{equation}

Next, we consider a closely related problem for tournaments. Recall that a tournament is a directed graph obtained by orienting an undirected complete graph, and a tournament is said to be transitive if there is a linear ordering of the vertices in which all the edges point in the same direction. The analogue of Ramsey's theorem for tournaments was proved by Erd\H{o}s and Moser~\citep{em} who showed that for each $t\in\N$, every tournament on $n$ vertices contains a transitive subtournament of order $t$ for all sufficiently large $n \in \N$.

As before, one cannot hope to find any non-transitive patterns in a general tournament, since the tournament in question might itself be transitive. Hence, in the spirit of Bollob\'as's question, Fox and Sudakov~\citep{fs} asked what non-transitive tournaments one is guaranteed to find in any tournament that is $\delta$-far from being transitive for some $\delta > 0$, where a tournament of order $n$ is said to \emph{$\delta$-far from being transitive} if the orientation of at least $\delta n^2$ of its edges need to be reversed in order to make it transitive. Consider the tournament with $3t$ vertices, which we call the \emph{unavoidable $t$-tournament}, formed by taking three disjoint transitive tournaments each of order $t$ on vertex sets $V_1$, $V_2$ and $V_3$, and directing edges from $V_i$ to $V_{i+1}$ for each $i = 0, 1, 2$, with indices being taken modulo $3$. As before, it is not hard to see that the only non-transitive patterns one could hope to find, even in a given tournament that is far from being transitive, are precisely such unavoidable tournaments (since the given tournament might itself be of this form). Fox and Sudakov~\citep{fs} showed that for any $t \in \N$ and $\delta > 0$, there exists a least integer $\TT(t,\delta) \in \N$ such that any tournament on $n$ vertices that is $\delta$-far from being transitive contains an unavoidable $t$-tournament for all $n \ge \TT(t,\delta)$. Turning to quantitative estimates, Long~\citep{long} subsequently determined the order of growth of the Ramsey number $\TT(t,\delta)$, showing in particular that $\TT(t,\delta) = (1/\delta)^{\Theta(t)}$.

Our second result pins down the off-diagonal growth rate of the Ramsey number $\TT(t,\delta)$ for each $t \ge 3$ as $\delta \to 0$. Our second theorem, again in its Tur\'an-type formulation, is as follows.
\begin{theorem}\label{tourthm}
	For each integer $t \ge 3$, there exists a $C = C(t)>0$ such that any tournament on $n \ge C$ vertices that is $Cn^{-1/\lceil t/2 \rceil}$-far from being transitive contains an unavoidable $t$-tournament.
\end{theorem}
Our second result is also tight up to the multiplicative factor $C(t)$ in its statement for each $t \ge 3$, again conditional on the truth of~\eqref{kstconj} in the case of $a = \lceil t/2 \rceil$ and $b = t$. Theorem~\ref{tourthm} hence determines, conditional on~\eqref{kstconj}, the order of magnitude of $\TT(t,\delta)$ when $t$ is much smaller than $1/\delta$; indeed, it follows from our result that, for each fixed $t \ge 3$ and in the limit of $\delta \to 0$, we have
\begin{equation}\label{tourbound}
	\TT(t,\delta) =  \Theta\left(\left(\frac{1}{\delta}\right)^{ \lceil t/2 \rceil}\right).
\end{equation}

Let us remark that while the statement of Theorem~\ref{tourthm} and the assertion about its sharpness bear a close resemblance, respectively, to the statement of Theorem~\ref{colourthm} and the assertion about its sharpness, we need to work significantly harder to establish these facts for tournaments than we do in the case of colourings.

The rest of this paper is organised as follows. We gather some preliminary facts together in Section~\ref{prelim}. We present the short proof of Theorem~\ref{colourthm}, as well as the simple construction demonstrating its sharpness, in Section~\ref{sec:col}. We postpone the more delicate proof of Theorem~\ref{tourthm}, as well as the construction demonstrating its sharpness, to Section~\ref{sec:tour}. Our main results address unavoidable patterns of order $t \ge 3$, and the case of $t=1$ is trivial; the exceptional case of $t = 2$ however demonstrates some anomalous behaviour, which we address in Section~\ref{sec:excep}. Finally, we conclude with a discussion of open problems in Section~\ref{conc}.

\section{Preliminaries}\label{prelim}
Here, we collect together the conventions we adopt when dealing with graphs, both directed and undirected, as well as a few useful results that we shall rely on.

Let $G=(V,E)$ be an undirected graph. We write $v(G)$ and $e(G)$ respectively for the number of vertices and edges of $G$. We denote the neighbourhood of a vertex $x \in V(G)$ by $N_G(x)$, and we write $\Deg_G(x)=|N_G(x)|$ for the degree of $x$ in $G$. More generally, for a set of vertices $X \subset V(G)$, the common neighbourhood $N_G(X)$ of $X$ is defined to be the set of vertices adjacent to all the vertices in $X$. Here and elsewhere, we omit the subscripts specifying the graph when the graph in question is clear from the context. Finally, for a set $X\subset V(G)$ of vertices, we write $G[X]$ for the subgraph of $G$ induced by $X$, and given two disjoint sets $X, Y \subset V(G)$, we write $G[X,Y]$ for the induced bipartite subgraph between the vertex classes $X$ and $Y$ in $G$.

For a fixed graph $H$, the Tur\'an number $\ex(n, H)$ is the maximum number of edges in an $n$-vertex graph with no subgraph isomorphic to $H$. It is known that the Kov\'ari--S\'os--Tur\'an conjecture~\eqref{kstconj} is, if true, tight up to multiplicative constants. The following fact, which may be deduced using the technique of dependent random choice, demonstrates this and a bit more; see~\citep{drc}, for example.
\begin{proposition}\label{drc1}
	For all $K, t \in \N$, there exists a $C = C(K,t) > 0$ such that any graph $G$ on $n$ vertices with $e(G) \ge Cn^{2-1/t}$ edges contains a set $S \subset V(G)$ of $K$ vertices in which each subset $X \subset S$ of $t$ vertices satisfies $|N(X)| \ge K$. \qed
\end{proposition}

It will also be convenient to record the fact that bipartite Tur\'an problems are `degenerate' in the following form.

\begin{proposition}\label{drc2}
	For all $t \in \N$ and $\eps > 0$, there exists a $C = C(t,\eps) > 0$ such that every bipartite graph $G$ between vertex classes $X$ and $Y$ with $|X|, |Y| \ge C$ and $e(G) \ge \eps |X||Y|$ contains a copy of $K_{t,t}$. \qed
\end{proposition}

Our notation when dealing with directed graphs mirrors our notation for undirected graphs: for instance, we write $N^+(\cdot)$ for out-neighbourhoods, $\Deg^-(\cdot)$ for in-degrees, and so on. We shall require some simple properties of orderings of tournaments that minimise the number of `backward edges'. Let $\sigma = (v_1, v_2, \dots, v_n)$ be an ordering of the vertex set of an $n$-vertex digraph $D$. A forward edge in $\sigma$ is an edge of $D$ of the form $v_i v_j$ with $i < j$, and a backward edge is any edge of $D$ that is not a forward edge. Given $i < j$, the interval or segment $[i, j]$ in $\sigma$ refers to the set of vertices $\{v_i, v_{i+1}, \dots, v_j\}$.  Also, for $i < j$, we say that the distance between $v_i$ and $v_j$ is $j-i$, the length of an edge then being the distance between its endpoints. Finally, given two disjoint sets $X, Y \subset V(D)$, we say $X$ precedes $Y$ in $\sigma$, or $X < Y$ in short, if each vertex of $X$ appears in $\sigma$ before every vertex of $Y$.

The following proposition, see~\citep{long}, follows easily from `switching' arguments.
\begin{proposition}\label{minorder}
	Let $T$ be an $n$-vertex tournament and let $\sigma = (v_1, v_2, \dots, v_n)$ be an ordering of $V(T)$ that minimises the number of backward edges. Then for any $1 \le i < j \le n$, we have
	\begin{enumerate}
		\item $|N^+(v_i) \cap [i+1,j ]| \ge (j-i)/2$,
		\item $|N^-(v_j) \cap [i,j-1 ]| \ge (j-i)/2$, and
		\item $\sigma$ minimises, when restricted to the interval $[i, j]$, the number of backward edges in $T[[i,j]]$. \qed
	\end{enumerate}
\end{proposition}

We also require the following variant of a lemma due to Long~\citep{long} that says, roughly speaking, that a tournament either has many `long backward edges' or a subtournament that is `further from being transitive'.

\begin{lemma}\label{longlemma}
Let $\alpha>0$, let $T$ be an $n$-vertex tournament that is $\alpha$-far from being transitive and suppose that $\sigma$ is an ordering of $V(T)$ that minimises the number of backward edges. Then $T$ either contains
\begin{enumerate}
	\item $\alpha n^2 / 1000$ backwards edges in $\sigma$ each of length at least $n/50$, or
	\item a subtournament of order at least $n/20$ that is $6\alpha$-far from being transitive.
\end{enumerate}
\end{lemma}
\begin{proof}
	Suppose that the number of backward edges in $\sigma$ of length at least $n/50$ is at most $\alpha n^2 / 1000$, so writing $F$ for the set of backward edges in $\sigma$ of length at most $n/50$, we have $|F|\ge (999/1000) \alpha n^2$. 
	
	Let $A$ be the initial segment of the first $n/20$ vertices in $\sigma$, and let $B$ be the terminal segment of the final $n/20$ vertices in $\sigma$. If there are more than $6\alpha(n/20)^2 $ backward edges within either $A$ or $B$, then we are done by Proposition~\ref{minorder}. Let $F' \subset F$ be the set of those edges with at most one endpoint in $A$ or in $B$, and note that $|F'|\geq |F|-12\alpha(n/20)^2 \ge (28/30) \alpha n^2$. 
	
	Now, each edge in $F'$ has length at most $n/50$ and does not lie entirely within either $A$ or $B$, so each edge of $F'$ lies entirely within the interval $[n/40, 39n/40]$. Since we have excluded edges at the extremes of $\sigma$ in $F'$, it is straightforward to check that a uniformly random interval of $n/20$ vertices contains both endpoints of an edge in $F'$ with probability at least $1/40$. Hence, there is an interval $I$ of $n/20$ vertices in $\sigma$ with at least $(1/40) (28/30) \alpha n^2$ backward edges within it. From Proposition~\ref{minorder}, it follows that $T[I]$ is at least $9\alpha$-far from being transitive, as required. 
\end{proof}

We shall make use of Ramsey's theorem in its various guises; see~\citep{book}, for instance.
\begin{proposition}\label{bipRam}
	For all $t\in \N$, there exist integers $\RR(t)$, $\BB(t)$ and $\TTT(t)$ such that the following hold. Every two-colouring of $K_n$ with $n \ge \RR(t)$ contains a monochromatic copy of $K_t$. Every two-colouring of $K_{n,n}$ with $n \ge \BB(t)$ contains a monochromatic copy of $K_{t,t}$. Every tournament on $n \ge \TTT(t)$ vertices contains a transitive subtournament on $t$ vertices.\qed
\end{proposition}

Finally, a word on asymptotic notation is also in order. We shall make use of standard asymptotic notation; the variable tending to infinity will always be $n$ unless we explicitly specify otherwise. When convenient, we shall also make use of some notation (of Vinogradov) that might be considered non-standard: given functions $f(n)$ and $g(n)$, we write $f \ll g$ if $f = O(g)$ and $f \gg g$ if $g = O(f)$. Here, constants suppressed by the asymptotic notation may depend on fixed parameters such as $t$, but not on $n$ or quantities depending on $n$ such as $\delta$.

\section{Colourings}\label{sec:col}
In this section, we deal with unavoidable colourings. We start by presenting an extremal construction complementing Theorem~\ref{colourthm}.
\begin{proposition}\label{coltight}
	For each integer $t \ge 3$ and all large enough $n \in \N$, there is a two-colouring of $K_n$ not containing an unavoidable $t$-colouring that is $\delta_t$-far from being monochromatic, where $\delta_t (n) =  (\ex(n,K_{t,t}) - 1)/2n^2$.
\end{proposition}

In particular, if the conjectural bound~\eqref{kstconj} holds, we actually have $\delta_t(n) \gg n^{-1/t}$ in the above construction for each $t \ge 3$.

\begin{proof}[Proof of Proposition~\ref{coltight}]
	Given $t\ge 3$, start with a graph $G$ on $[n]$ with $m = \ex(n,K_{t,t})-1$ edges that does not contain any copies of $K_{t,t}$, and pass to a bipartite subgraph $H$ of $G$ with at least $m/2$ edges. Now, colour the edges of the complete graph on $[n]$ by colouring all the edges of $H$ red, and all the other edges blue. The construction ensures that there is no red clique on three vertices, and that there are no red copies of $K_{t,t}$. Since $m=o(n^2)$, clearly the number of both the red edges and the blue edges is at least $m/2$ provided $n$ is sufficiently large, so the claim follows.
\end{proof}

Having demonstrated its sharpness, we now give the proof of Theorem~\ref{colourthm}.

\begin{proof}[Proof of Theorem~\ref{colourthm}]
	Let us fix $C = C(t)$ to be large enough to support all of the estimates that follow, and suppose that we have a two-coloured complete graph $G$ on $n \ge C$ vertices in which the number of both the red and the blue edges is at least $Cn^{2-1/t}$. We assume that $G$ is a counterexample to the result that does not contain an unavoidable $t$-colouring and thereby derive a contradiction.

	We denote the graphs spanned by the red and blue edges of $G$ by $R$ and $B$ respectively. We also assume, without loss of generality, that there are at least as many blue edges as there are red edges in $G$.

	The first step in the proof is to show that we may remove a very small number of vertices from $G$ so that there are no red copies of $K_t$ in the resulting graph.
	\begin{claim}\label{claim:fewredvertices}
		For every $\eps>0$, there exists $C_1=C_1(t,\eps)$ so that we may find a set $S \subset V(G)$ of size at most $C_1$ such that in $G' = G[V(G) \setminus S]$, every vertex is incident to at least $(1-2\eps)n'$ blue edges, where $n' = v(G')$.
	\end{claim}
	\begin{proof}
		We fix $C_1$ to be large enough, with the benefit of hindsight, to support the argument that follows.

		If the set $S_r$ of vertices $x$ with $\Deg_R(x) \ge \eps n$ has size at most $C_1$, the claim follows by taking $S = S_r$. Hence, assume that $|S_r| \ge C_1$. As blue is the most common colour, we know that $e(B) \ge \binom{n}{2} / 2$, so there are at least $n/8$ vertices $x$ for which $\Deg_B(x) \ge n/4$; let $S_b$ be a set of $C_1$ such vertices disjoint from $S_r$.

		We claim that there is an $m = m(t, \eps)$ such that $S_r$ does not contain any blue copies of $K_m$, and such that $S_b$ does not contain any red copies of $K_m$. To see this, assume that $m$ is sufficiently large and that $X\subset S_r$ induces a blue copy of $K_m$. Then, provided $m$ is large enough, Propostion~\ref{drc2} implies that there is a subset $Y\subset X$ of order $t$ for which the common neighbourhood $|N_R(Y)|\ge \RR(t)$. Ramsey's theorem applied to $N_R(Y)$ now shows that there is an unavoidable $t$-colouring in $G$, a contradiction. Hence, we may assume that $S_r$ does not induce any blue copies of $K_m$ in $G$. The same argument, with the colours interchanged, allows us to assume that $S_b$ does not induce any red copies of $K_m$ in $G$.

		Now, from the bipartite form of Ramsey's theorem applied to the complete bipartite graph between $S_r$ and $S_b$, we may find $S'_r \subset S_r$ and $S'_b \subset S_b$, both of order $\RR(m)$, such that the complete bipartite graph between $S'_r$ and $S'_b$ is monochromatic. Applying Ramsey's theorem to each of $S'_r$ and $S'_b$ (combined with our earlier observation), we find a red copy of $K_m$ inside $S'_r$ and a blue copy of $K_m$ inside $S'_b$, which together yield an unavoidable $t$-colouring in $G$, a contradiction.
	\end{proof}

	We apply the previous claim with $\eps=1/(10t)$ to pass to a two-coloured complete graph $G'$ where all the vertices are incident to many blue edges, and as before, we denote the graphs spanned by the red and the blue edges of $G'$ by $R'$ and $B'$ respectively. 
	
	\begin{claim}
		There are no red copies of $K_t$ in $G'$.
	\end{claim}

	\begin{proof}
		Suppose that $X$ forms a red clique on $t$ vertices in $G'$. Since $\Deg_{B'}(x) \ge (1-2\eps)n'$ for each $x\in X$, it follows that $|N_{B'}(X)| \ge n' - t(n'/10t) \ge n'/2 \ge \RR(t)$. By applying Ramsey's theorem to $N_{B'}(X)$, we find an unavoidable $t$-colouring in $G$, a contradiction.
	\end{proof}

	Observe that the number of red edges in $G'$ is at least $C_2 (n')^{2-1/t}$ for some $C_2 = C_2(t)$, since we have removed at most $C_1$ vertices and $C_1 n$ red edges in passing from $G$ to $G'$. Provided $C_2$ is large enough, it is easy to see that there are many red copies of $K_{t,t}$ in $G'$. 
	
	The second step in the proof is to find a reasonably `well distributed' collection of such copies that we may use to produce an unavoidable $t$-colouring. If $C_2$ is large enough, then it follows from Proposition~\ref{drc1} that there is a set $Y\subset V(G')$ of size at least $C_3 = C_3(t)$ such that every $X\subset Y$ of order $t$ satisfies $|N_{R'}(X)| \ge \RR(t)$.

	\begin{claim}\label{claim: rededges}
		There are no blue copies of $K_t$ in $G'[Y]$.
	\end{claim}
	\begin{proof}
		If $X \subset Y$ induces a blue clique on $t$ vertices in $G$', then since $|N_{R'}(X)| \ge \RR(t)$, the previous claim combined with Ramsey's theorem allows us to find a blue clique of order $t$ within $N_{R'}(X)$, and consequently, an unavoidable $t$-colouring in $G$; again, we have a contradiction.
	\end{proof}
	\begin{claim}
		$Y$ induces at least $|Y|^2/t^2$ red edges in $G'$.
	\end{claim}
	\begin{proof}
		Indeed, by the previous claim, every subset of $Y$ of order $t$ contains at least one red edge, and each red edge belongs to at most $\binom{|Y|}{t-2}$ such sets. Hence, $Y$ induces at least $\binom{|Y|}{t}/ \binom{|Y|}{t-2} \ge |Y|^2/t^2$ red edges in $G'$.
	\end{proof}

	We may now finish as follows. If $C_3$ is large enough, then it follows from the previous claim and Proposition~\ref{drc2} that $Y$ contains a red copy of $K_{\RR(t),\RR(t)}$. Applying Ramsey's theorem to each of the vertex classes of such a red copy of $K_{\RR(t),\RR(t)}$ in $Y$ and utilising the claims above, we find an unavoidable $t$-colouring in $G$, contradicting our initial assumption that $G$ is a counterexample to the result.
\end{proof}

\section{Tournaments}\label{sec:tour}
In this section, we deal with unavoidable tournaments. In what follows, to save space, we write $\DD_t$ to denote the unavoidable $t$-tournament on $3t$ vertices. As before, we start by presenting an extremal construction complementing Theorem~\ref{tourthm}.

\begin{proposition}\label{tourtight}
	For each integer $t \ge 3$ and all large enough $n \in \N$, there is a tournament on $n$ vertices not containing $\DD_t$ that is $\delta_t$-far from being transitive, where $\delta_t (n) =  10^{-6}\ex(n,K_{\lceil t/2 \rceil,t})/n^2$.
\end{proposition}
Again, if the conjectural bound~\eqref{kstconj} holds, then this tells us that we actually have $\delta_t(n) \gg n^{-1/\lceil t/2 \rceil}$ in the above construction for each $t \ge 3$. While the construction demonstrating the above proposition is analogous to the construction for colourings presented earlier in Section~\ref{sec:col}, the argument justifying this construction is somewhat more involved.
\begin{proof}[Proof of Proposition~\ref{tourtight}]
	Given $t\ge 3$, we first set $r = r(t) = \lceil t/2\rceil \ge 2$. To prove the result, we shall now construct, for all large enough $n \in \N$, a tournament $T$ on $n$ vertices not containing any copies of $\DD_t$ with at least $\ex(n, K_{r,t}) / 10^{6}$ backward edges in any ordering of its vertex set. As usual, we assume that $n$ is large enough to support the estimates that follow. Note that since $t\ge r \ge 2$, we have $\ex(n, K_{r,t}) \gg n^{3/2}$; see~\citep{survey}, for example.

	Let $G$ be an $n$-vertex graph with $\ex(n,K_{r, t}) - 1$ edges which does not contain a copy of $K_{r, t}$ and let $H\subset G$ be a spanning bipartite subgraph of $G$ with at least $e(G)/2$ edges with vertex classes $A$ and $B$. We construct a tournament $T$ on the same vertex set as $H$ as follows: fix an ordering $\sigma$ of the vertices of $H$ where all vertices of $A$ precede all the vertices of $B$, and for every edge $xy\in E(H)$ with $x\in A$ and $y\in B,$ we direct the corresponding edge in $T$ backwards in $\sigma$ from $y$ to $x$ in $T$, and every other edge forwards. In what follows, we speak about the edges in $H$ and the edges in $T$ directed from $B$ to $A$ interchangeably, since these are in one-to-one correspondence with each other. 

	It is not hard to see that $T$ does not contain any copies of $\DD_t$, a fact that we record below.
	\begin{claim}
		$T$ does not contain any copies of $\DD_t$.
	\end{claim}
	\begin{proof}
		Suppose to the contrary that there is a copy of $\DD_t$ in $T$, and let $X, Y, Z \subset V(T)$ be the three transitive vertex classes of this copy of $\DD_t$ in $T$, with edges oriented from $X$ to $Y$, from $Y$ to $Z$, and from $Z$ to $X$.

		Observe that it cannot happen that each of $X$, $Y$ and $Z$ meet $A$, since this would yield a cyclic triangle in $A$, while $T[A]$ is transitive by construction; hence, suppose without loss of generality that $Z \subset B$. The same argument applied to $B$ shows that one of $X$ or $Y$ must necessarily be contained in $A$; since there are no copies of $K_{r,t}$ in $H$ (which specifies the set of edges directed from $B$ to $A$), it must be the case that $Y \subset A$.

		We now know that $Y \subset A$ and $Z \subset B$. Of course, either $|X \cap A| \ge \lceil t/2 \rceil = r$ or $|X \cap B| \ge r$. If the former happens, then we find a copy of $K_{r, t}$ between $X \cap A$ and $Z$ in $H$, and if the latter happens, then we find a copy of $K_{r, t}$ between $X \cap B$ and $Y$ in $H$, a contradiction regardless.
	\end{proof}

	The bulk of the work, which we accomplish in the next claim, lies in demonstrating that $T$ is not too close to being transitive.
	\begin{claim}
		In any ordering of $V(T)$, there are at least $e(H)/10^{5}$ backward edges.
	\end{claim}

	\begin{proof}
		Suppose this does not hold, and let $\tau = (v_1,v_2,\ldots, v_n)$ be an ordering of the vertices of $T$ that minimises the number of backward edges, so that the number of backward edges in $\tau$ is less than $e(H)/10^5$. The basic idea now is simple. Since many of the backward edges in $\sigma$ are forward edges in $\tau$, we expect to be able to find large sets $A' \subset A$ and $B' \subset B$ such that $B'$ precedes $A'$ in $\tau$ with $H[A', B']$ containing a positive fraction of the edges in $H$. However, we may then use the fact that $H$ has no copies of $K_{r, t}$ to conclude that at least half of the edges, say, between $A'$ and $B'$ in $T$ must be directed backward from $A'$ to $B'$ in $\tau$. While this sketch is conceptually straightforward, filling in the details however necessitates dealing with some technicalities. 
		
		Let us start by recording the following fact that we shall make use of repeatedly.
		\begin{claim}\label{manyedges}
		If $A' \subset A$ and $B' \subset B$ are such that every vertex of $B'$ precedes every vertex of $A'$ in $\tau$, then $e(H[A', B']) \le e(H)/10^5$.
		\end{claim}
		\begin{proof}
		Suppose the claim fails for some $A' \subset A$ and $B' \subset B$. Since $e(H) \gg n^{3/2}$, we must have $|A'|, |B'| \gg \sqrt{n}$. Now, the graph $H[A', B']$ contains no copies of $K_{r,t}$, and both $|A'|$ and $|B'|$ are sufficiently large, so it follows from Proposition~\ref{drc2} that $e(H[A', B']) \le |A'||B'|/2$. Consequently, each non-edge of $H[A', B']$, of which there are at least $e(H)/10^5$, is directed from $A'$ to $B'$ in $T$. This yields at least $e(H)/10^5$ backward edges in $\tau$, which is a contradiction.
		\end{proof}
		
		Write $|A| = a$ and $|B| = b$ so that $a+b = n$, and define $X$ and $Y$ to be the first $a$ vertices and the last $b$ vertices in $\tau$ respectively. Let $S_B = B \cap X$ be those vertices of $B$ appearing in the first $a$ vertices in $\tau$, and let $S_A = A \cap Y$. Of course, we have $|S_B| = |S_A|$; we write $m$ for their common size.

		Observe that since there are at most $e(H)/10^5$ backward edges in $\tau$, at least $e(H)/2$ edges of $H$ are forward edges in $\tau$; each such forward edge must necessarily be directed out of some vertex in $S_B$ or into some vertex in $S_A$ (or both). Hence, we assume by pigeonholing that $F \subset E(H)$ is some set of $e(H)/4$ edges of $H$ that are all forward edges in $\tau$ directed into some vertex in $S_A$ (the other case being symmetric). Note that it must be the case that $m \gg n^{1/2} \ge n^{1/10}$, say, since we know that $mn \gg  e(H) \gg n^{3/2}$.

		We may partition $F$ as $F = F' \cup F''$, where $F'$ consists of all the forward edges in $\tau$ directed from $B \cap Y$ to $S_A$, and $F''$ consists of all those forward edges in $\tau$ directed from $S_B$ to $S_A$. We know from Claim~\ref{manyedges} that $|F''| \le |F| /2$, so we must have $|F'| \ge |F|/2 \ge e(H)/8$. 
		
		We now need the notion of a `balanced interval'. We fix a sufficiently large constant $C_1 = C_1(t) > 0$ to support what follows, and say that a sub-interval $W \subset Y$ is \emph{balanced} if
		\begin{enumerate}[label = {\bfseries A\arabic{enumi}}]
			\item\label{A1} the number of edges of $F'$ within $W$ is at least  $|F'|/2$,
			\item\label{A2} $1/50< |W\cap A|/|W\cap B| <50$, and
			\item\label{A3} either the initial segment $W^+_i\subset W$ of the first $i$ vertices in $W$ satisfies $|W^+_i\cap A| <50 |W^+_i\cap B|$ or the the terminal segment $W^-_i\subset W$ of all but the first $i$ vertices in $W$ satisfies $|W^-_i\cap B| <50 |W^-_i\cap A|$ for each $C_1 \le i\le |W| - C_1$.
		\end{enumerate}
		We first show that a balanced interval may always be found. Starting with the interval $Y$, we shall successively refine the interval under consideration into a `more structured' sub-interval, repeating this iteratively until we reach our goal. We start with $W_0 = Y$ and in each step $0 \le j \le 30\log n$, we do the following. If the interval $W_j$ is balanced, then we stop. If not, then we shall find a sub-interval $W_{j+1}\subset W_{j}$ with $|W_{j+1}|\le 9|W_{j}|/10$ that contains all but $C_2n$ edges of $F'$ within $W_j$, for some $C_2 = C_2(t) > 0$. We claim that such an iterative process must terminate in a balanced interval. Indeed, if the process does not terminate within the first $30\log n$ steps, then as we have lost at most $30C_2 n \log n$ edges of $F'$, the number of surviving edges from $F'$ is at least $|F'| - 30C_2n\log n \ge 3|F'|/4 \gg n^{3/2}$ since $|F'| \gg n^{3/2}$, while on the other hand, the number of surviving vertices is $O(1)$, which is clearly impossible.

		We now describe how to construct $W' = W_{j+1}$ from an unbalanced $W = W_{j}$ at some stage $0 \le j < 30 \log n$. We may assume inductively, as we saw earlier, that the number of edges of $F'$ with both endpoints in $W$ is at least $3|F'|/4 \gg n^{3/2}$, so in particular, we have $|W| \gg n^{3/4}$. Since $W$ is not balanced, it must violate one of~\ref{A2} or~\ref{A3}. We now describe how to construct $W'$ in the case where $W$ violates~\ref{A2} on account of $|W\cap B|>50|W\cap A|$, and then indicate the minor modifications needed to handle the other cases.

		Consider set $Z$ of the last $C_2$ vertices from $A$ in $W$ and suppose that $v_{k+1}$ is the first vertex in $Z$. We claim that we may take $W' = W^+_k$ to be the initial segment of those vertices preceding $v_{k+1}$ in $W$.

		To show that this choice of $W'$ works, we first claim that $|W' \cap B| < 4/5 |W \cap B|$. Indeed, if this is not the case, we may find a copy of $K_{r, t}$ in $H$ by arguing as follows. Given a vertex $ v = v_{i+1} \in Z$, since $\tau$ is an ordering that minimises the number of backward edges, we know from Proposition~\ref{minorder} that at least $1/2$ of the edges between $W^+_i$ and $v$ are directed into $v$. Since
		\[|W' \cap B| \ge 4|W \cap B|/5 \ge(4/5)(50/51) |W| \ge 3 |W|/4,\] the number of edges directed from $W' \cap B$ to $v$ is at least \[|W' \cap B| - |W^+_i|/2 \ge 3|W|/4 - |W|/2 = |W|/4.\]  It then follows that the number of edges directed from $W' \cap B$ to $Z$, all necessarily edges in $H$, is at least $|W' \cap B| |Z|/4$. Therefore, as $|W' \cap B| \ge 3|W|/4 \gg n^{3/4}$ and $|Z| = C_2$, then provided $C_2$ is suitably large, we conclude from Proposition~\ref{drc2} that there is a copy of $K_{r, t}$ in $H$ between $W' \cap B$ and $Z$, a contradiction.

		If $W$ violates~\ref{A2} on account of $|W\cap A|>50|W\cap B|$, then we analogously construct $W'$ by considering the first $C_2$ vertices from $B$ in $W$. Finally, if $W$ violates~\ref{A3} for some $C_1 \le i \le |W|-C_1$, then we apply the above argument to both $W^+_i$ and $W^-_i$, looking at the first $C_2/ 2$ vertices from $B$ in the former interval, and the last $C_2 / 2$ vertices from $A$ in the latter interval.

		To finish the proof of the claim, we shall show that the existence of a balanced interval $J \subset Y$ yields too many backward edges in $\tau$. We need a little notation: for a partition	of $J= J_1 \cup J_2$ into an initial segment $J_1$ and a terminal segment $J_2$, we decompose the subset of at least $|F'|/2$ edges of $F'$ within $J$ into three parts as $F_1 \cup F_2 \cup F_{12}$, were $F_i \subset F'$ is the set of such edges entirely within $J_i$ for $i = 1, 2$ and $F_{12} \subset F'$ is the set of such edges directed from $J_1$ to $J_2$.
		
		Now, fix $J_1$ to be the smallest initial segment of $J$ for which $|F_1| \ge |F'|/4$ and set $J_2 = J \setminus J_1$. Since $|F'| \gg n^{3/2} \gg n$, it follows that $|F'| /4 \le |F_1| \le |F'|/3$. We cannot have $|F_{12}| \ge |F'|/10$, since this would imply that there are too many edges directed from $J_1 \cap B$ to $J_2 \cap A$, contradicting Claim~\ref{manyedges}. Thus, we may assume that $|F_1|, |F_2| \ge |F'|/4$. It must be the case that $|J_1\cap B| \le |J\cap B|/1000$, for if not, then we would have
		\[|J_1\cap B|| J_2 \cap A|\ge  |J_2\cap B||J_2\cap A|/1000 \ge |F_2|/1000 \ge |F'|/4000 \ge 2e(H)/10^5,\]
		which when combined with Claim~\ref{manyedges} promising us that the number of edges from $J_1 \cap B$ to $J_2 \cap A$ is at most $e(H)/10^5$, yields at least $e(H)/10^5$ backward edges in $\tau$, a contradiction. The same reasoning also leads us to conclude that $|J_2\cap A|\le |J\cap A|/1000$. These two assertions taken together contradict the fact that $J$, being balanced, satisfies~\ref{A3}; the claim now follows.
	\end{proof}

	We have shown that the tournament $T$ we constructed has both the properties we desire, completing the proof of the proposition.
\end{proof}

Let us introduce some conventions that we adopt in the sequel. In what follows, given a tournament $T$, we shall work exclusively with an ordering $\sigma$ of its vertex set minimising the number of backward edges, so all subsequent references to forward or backward edges, intervals of vertices, lengths of edges, etc., will be with respect to this ordering. Given two disjoint sets of vertices $A < B$ of a tournament $T$, we define $d(A,B)$ to be the distance between the largest vertex of $A$ and the smallest vertex of $B$, and we abuse notation slightly and define $d(A)$ to be the distance between the smallest vertex and the largest vertex of $A$. With this language in place, we are now ready to prove our second main result.

\begin{proof}[Proof of Theorem~\ref{tourthm}]   

	We start by fixing $t\ge 3$ and setting $r = r(t) = \lceil t/2\rceil \ge 2$, and we take $C = C(t) > 1$ to be large enough to support the argument that follows. 
	
	Our argument will be by contradiction. Starting with a tournament on $N_0$ vertices that is at least $CN_0^{-1/r}$-far from being transitive with no copy of $\DD_t$, we repeatedly Lemma~\ref{longlemma} until we reach a subtournament $T$ on $N$ vertices which has an ordering $\sigma$ of its vertex set minimising the number of backward edges in which at least $CN^{2-1/r}$ backward edges have length least $N/50$. Furthermore, we may of course suppose that $N$ is large enough to support the arguments that follow. 
	
	We justify the above claim as follows. After $k$ unsuccessful applications of Lemma~\ref{longlemma}, we are left with a tournament on $n = N_0/20^{k}$ vertices, whose distance from being transitive is at least $6^k C n^{-1/r}$, so the number of backward edges in any ordering of such a tournament is at least 
	\[6^k \cdot CN_0^{-1/r} \cdot n^2 = (6/20^{1/r})^k \cdot C n^{2-1/r}\] 
Now, we know $r \ge 2$, so $6/20^{1/r} \ge 6/\sqrt{20} > 1$, so it follows that we must have a successful application of Lemma~\ref{longlemma} before $n$ becomes too small, since the number of backward edges is both at least $Cn^{2-1/r}$ and at most $n^2/2$, and indeed, if we start with $C$ large enough, we may assume that we succeed at a stage where $n$ is sufficiently large.
	
	In what follows, we shall work with $T$, which is a tournament on $N$ vertices with no copy of $\DD_t$. Furthermore, we also fix $\sigma$, an ordering of $V(T)$ minimising the number of backward edges with respect to which we know that there are at least $C N^{2-1/r}$ backward edges of length at least $N/50$.
	
	First, we need an analogue of a lemma we used in the context of finding unavoidable colourings that allows us to deal with vertices of atypically large degree (with respect to the backward edges).
 
    \begin{claim}\label{tourdegrees}
    	For any $\eps > 0 $, there exist positive integers $C_1 = C_1(t, \eps) > 0$ and $C_2 = C_2(t, \eps) > 0$ such that the following holds. For any interval $I$ in $\sigma$ of at least $n \geq C_1$ vertices, the induced tournament $T[I]$ contains at most $C_2$ vertices that are incident with more than $\eps n$ backward edges in $\sigma$.
    \end{claim}
    
    \begin{proof}
    	We argue by contradiction, always ensuring that the numbered constants $C_1, C_2, C_3, \dots$ in our argument are sufficiently large as a function of $t$. Suppose there exists a set $S\subset I$ consisting of $C_2$ vertices each sending out at least $\eps n$ backward edges in $\sigma$; the other case where these vertices receive many backward edges may be handled analogously. 
	
	First, by pigeonholing, we pass to a large subset $S'\subset S$ of order at least $C_3 = C_3(t) > 0$ with the property that $d(S')\leq \eps n/100$. Now, consider the subsegment $I'$ of $I$ consisting of exactly $\eps n/ 10$ vertices, immediately to the left of the first vertex of $S'$ in $\sigma$. Each vertex of $S'$ sends out at least $\eps n$ backward edges in $\sigma$, so each such vertex sends at least $\eps n/2$ backward edges to the left of $I'$ to the interval $I''$, where $I''$ is the subinterval of $I$ preceding $I'$ in $\sigma$. 
	
	Next, we may find a large set of at least $C_4 = C_4 (t) > 0$ vertices in $S'$ with a large common out-neighbourhood in $I''$; more precisely, by selecting a random subset of $S'$ and appealing to convexity, we find a set $A\subset S'$ of size $C_4$ with at least $(\eps/3)^{C_4} n$ out-neighbours in $I''$; call this set of common out-neighbours $B^{+}$. By Proposition~\ref{minorder}, we know that each vertex in $A$ receives at least $\eps n/30$ edges from $I'$. Again, by the same argument, we may pass to a large subset $A'\subset A$ of size at least $C_5 = \TTT(t)$ having at least $(\eps/3)^{C_5}n$ common in-neighbours in $I'$; call this set of common in-neighbours $B^{-}$. 
	
	We may now finish as follows. If there is a $K_{\TTT(t),\TTT(t)}$ directed from $B^{+}$ to $B^{-}$, then by passing to transitive tournaments within the partite classes of this copy and within $A'$, we may find a $\DD_t$ in $T$, which is a contradiction. Therefore, by evoking Proposition~\ref{drc2}, we conclude that there are at least $n^2 / C_6$ backward edges between $I'' \cup I'$ for some $C_6 = C_6(t) > 0$. Since $I'' \cup I'$ is an interval in $\sigma$, we conclude that $T[I'' \cup I']$ has distance at least $1/C_6$ from being transitive, and therefore contains a copy of $\DD_t$ provided $N'$ is sufficiently large, another contradiction.
	\end{proof}
    
    We apply the previous lemma with $\eps=1/(100t^2)$ to the entire tournament $T$, concluding that we may remove $O(1)$ vertices from $T$ and guarantee that in the resulting tournament $T'$ on $m$ vertices, no vertex is incident to more than $2\eps m$ backward edges in the ordering induced by $\sigma$ on $V(T')$, and that $T'$ has at least $Cm^{2-1/r}$ backward edges of length at least $m/20$ with respect to $\sigma$. 
    
    In what follows, we work with $T'$ and the ordering induced by $\sigma$ on $V(T')$, though we abuse notation slightly and refer to this induced ordering as $\sigma$ as well. We call a backward edge of $T'$ \emph{good} if its length is at least $m/20$; of course, we know that $T'$ has at least $Cm^{2-1/r}$ good backward edges.

    \begin{claim}\label{noktt}
    	We may assume that $T'$ does not contain a copy of $K_{t,t}$ formed from good backward edges, where one partite class of this copy precedes the other in $\sigma$, and each of the partite classes of this copy forms a transitive subtournament.
    \end{claim}
    \begin{proof}
	Suppose that such a copy exists, say with partite classes $A$ and $B$ with $A<B$ where both $T[A]$ and $T[B]$ are transitive. We know there are at least $m/20$ vertices between the last element of $A$ and the first of $B$; call this intervening interval $P$. Since we removed all vertices incident to many backward edges in passing to $T'$, we know that $A$ has at least $|P|-2tm/(100t^2)  \ge 2|P|/3$ common out-neighbours in $P$, and that $B$ similarly has at least $2|P|/3$ common in-neighbours in $P$. Provided $m$ is large enough, we can then find a set $S \subset P$ of $\TTT(t)$ vertices in $P$ where all the edges are directed from $A$ to $S$ and from $S$ to $B$. Passing to a transitive subtournament inside $S$, we find a copy of $\DD_t$ in $T'$, which is a contradiction. 
    \end{proof}
    
    Now, we partition $V(T')$ into $100$ intervals of size $m/100$, and observe that at least a $1/10^4$ fraction of the good backward edges of $T'$ lie between two of these intervals; call these intervals $I$ and $J$ with $I < J$, and note that we necessarily have $d(I,J)\geq m/30$. To summarise, we now have two intervals $I< J$ of order $m/100$ for which there exists at least $C'm^{2-1/r}$ good backward edges directed from $J$ and $I$, for some large $C' = C'(t) > 0$.
    
    Next, we shall find two disjoint sets of edges from the good backward edges between $I$ and $J$ in such a way that every pair of edges across these two sets interlace nicely. To do so, we need to prepare $I$ and $J$ appropriately. We know that the number of good backward edges between $I$ and $J$ is at least $C'(|I|+|J|)^{2-1/r}$, so we pass to subintervals $I' \subset I$ and  $J' \subset J$ chosen such that the number of good backward edges from $J'$ to $I'$ is of the form $K(|I'|+|J'|)^{2-1/r}$ with $K \ge C'$ maximal.
   
    Note that the sizes of $I'$ and $J'$ are comparable; indeed, we must have $|J'|/3 <|I'|<3|J'|$. Suppose this does not hold, and assume $|J'|\geq |I'|=q$. Writing $|J'|=p\cdot q+s$, we may partition $J'$ into $p$ consecutive intervals of size $q$ and one interval of size $s\leq q$. By the maximality of $K$, the number of good backward edges between $I'$ and any interval in this decomposition of $J'$ is at most $K(2q)^{2-1/r}$, so the total number of good backward edges from $J'$ to $I'$ is at most $(p+1)\cdot K(2q)^{2-1/r}< K((p+1)q)^{2-1/r}<K(|I'|+|J'|)^{2-1/r}$, provided $p\geq 3$ and $r\geq 2$. 
    
    We now find an appropriate collection of interlacing edges between $I'$ and $J'$ using the following density increment argument.
    \begin{claim}\label{densityincr}
    	For every $\alpha\in\ [0,1/2]$ and $\eps>0$, there exists $ C_7 = C_7 (\alpha, \eps) >0$ such that the following holds. Fix a tournament $T$ and an ordering of its vertices. For any two intervals $I < J$ of vertices with $|I|+|J|=n$ for which there are $L\geq C_7 n^{2-\alpha}$ backward edges from $J$ to $I$ the following holds: either there exists a partition of $I$ into two intervals $I=I_1 \cup I_2$ and a partition of $J=J_1 \cup J_2$ into two intervals with $I_1<I_2$ and $J_1<J_2$ such that the number of backward edges from $J_1$ and $I_1$ and from $J_2$ and $I_2$ is at least $\eps L$, or there exist two intervals $I'\subset I$ and $J'\subset J$ with $|I'|+|J'|\leq n/2$ where the number of backward edges from $J'$ and $I'$ is at least $(1/2-3\eps)L$.
    \end{claim}
    
    \begin{proof} We argue as follows. Consider the smallest initial segment of $I$, say $I_1$, such that the number of backward edges with an endpoint in $I_1$ is at least $L/2$. Since $L=\omega(n)$, the set $I\setminus I_1=I_2$ must also be incident with at least $L/2 + o(L)$ backward edges. 
    
    Now, enumerate $J = \{1, 2,\dots, |J|\}$ and for each $j\in J$, let us define $b(j)$ to be the difference between the number of backward edges between $I_1$ and $\{1,\ldots, j-1\}$ and the number of backward edges between $I_2$ and $\{j,\ldots, |J|\}$. We know that $b(1) = -L/2 + o(L)$ and $b(|J|) = L/2 + o(L)$, and since $b(t+1)=b(t)+o(L)$, there exists a vertex $p\in J$ such that $b(p)=o(L)$; accordingly, let $J_1 = \{1,\ldots, p-1\}$ and $J_2 = \{p,\ldots, |J|\}$.
    
    Suppose that the number of backward edges between $I_1$ and $J_1$ is at least $2\eps L$. Then there must also exist $2\eps L + o(L) \ge \eps L$ backward edges between $I_2$ and $J_2$, in which case, we are done. If the above assumption does not hold, then the number of backward edges between $I_1$ and $J_2$ and the number of backward edges between $I_2$ and $J_1$ are both at least $(1/2-3\eps)L$; we find $I'$ and $J'$ by now taking the pair with the smaller total size, proving the claim.
    \end{proof}
    
We now apply the above claim to $I'$ and $J'$ to find many interlacing good backward edges between them. Indeed, we apply the previous claim (with $\alpha=1/r$ and $\eps=1/100$) to the backward edges between $I'$ and $J'$. If the latter conclusion of the claim holds, then we find two intervals $I''\subset I'$ and $J''\subset J'$ such that the number of backward edges between $I''$ and $J''$ is at least $ (1/2-1/50)K(|I'|+|J'|)^{2-1/r}$. However $|I''|+|J''| \leq (|I'|+|J'|)/2$, from which it follows that 
\[(1/2-1/50)K(|I'|+|J'|)^{2-1/r} \geq (3 K / 2) (|I''|+|J''|)^{2-1/r},\] which contradicts the maximality of $K$. So the former conclusion of the claim must hold, which means that we may find four intervals $I_1<I_2<J_1<J_2$ where $|I_1|+ |I_2|=|I'|$ and $|J_1|+|J_2| = |J'|$ such that the number of good backward edges between $I_1$ and $J_1$ is at least $(K/400)(|I_1|+|J_1|)^{2-1/r}$ and the number of good backward edges between $I_2$ and $J_2$ is at least  $(K/400)(|I_2|+|J_2|)^{2-2/t}$. Moreover, as we already know, the distance between $I_2$ and $J_1$ is at least $m/30$. 

In the light of the above discussion, let us select a collection of four intervals $I_1<I_2<J_1<J_2$ with $d(I_2,J_1)\geq m/30$ such that the numbers of good backward edges between $I_1$ and $J_1$ and between $I_2$ and $J_2$ are respectively at least $K_1(|I_1|+|J_1|)^{2-1/r}$ and $K_2(|I_2|+|J_2|)^{2-1/r}$ with $K_1 + K_2>0$ as large as possible. In what follows, we argue under the assumption that $|I_2|\geq |I_1|$ and $|J_2|\geq |J_1|$; the three other cases may be handled analogously.

    Just as we have some separation between $I_2$ and $J_1$, it will also be convenient to introduce some separation between the other intervals. For this purpose, we shall subdivide $I_2$ and $J_2$ into two new intervals. Indeed, let $I_2=I_3\cup I_4$ where $I_3$ is the initial segment of $I_2$ order $|I_2|/100$, and similarly let $J_2=J_3\cup J_4$ where $J_3$ is the initial segment of $J_2$ of order $|J_2|/100$. Let $|J_2|=p$ and $|I_2|=p'$, and assume $p'\leq p$. Using the same line of reasoning that established that $|I'|$ and $|J'|$ were comparable, we may suppose that $p/3\leq p' \leq p$. Finally, using the maximality of $K_2$ once $I_1$ and $J_1$ (and hence $K_1$) are fixed, we see that the number of good backward edges between $I_4$ and $J_4$ is at least 
    \[K_2 (p+p')^{2-1/r}- K_2(p'+p/100)^{2-1/r}-K_2(p+p'/100)^{2-1/r},\]
which may be checked to be at least $(K_2/100)(|I_4|+|J_4|)^{2-1/r}$. 
     
     In the rest of the argument, we work exclusively with these four intervals $I_1<J_1<I_4<J_4$. Let us summarise what we need about these intervals below.    
    \begin{enumerate}
    	\item $I_1<I_4<J_1<J_4$.
    	\item $d(I_4, J_1) \geq m/30$.
    	\item $d(I_1,I_4)\geq (|I_1|+|I_4|)/200$.
    	\item $d(J_1,J_4)\geq (|J_1|+|J_4|)/200$.
    	\item The number of good backward edges between $I_1$ and $J_1$ is at least $K'(|I_1|+|J_1|)^{2-1/r}$, for some suitably large constant $K'>0$.
    	\item The number of good backward edges between $I_4$ and $J_4$ is at least $K'(|I_4|+|J_4|)^{2-1/r}$, for some suitably large constant $K'>0$.
    \end{enumerate}
    We need to rule out the possibility of finding a copy of $K_{t,t}$ in the backward edges between either $I_1$ and $I_4$ or $J_1$ and $J_4$. We may accomplish this as before. We first observe that $|I_i|,|J_i|$ with $i\in \{1,4\}$ must all be sufficiently large in order to satisfy these properties since $K'$ may be assumed to be sufficiently large. Moreover, by mimicking the argument used to prove Claim~\ref{tourdegrees} and using the fact that we have only deleted $O(1)$ vertices so far, we may delete $O(1)$ further vertices from $I_1\cup I_4$ so that none of remaining vertices are incident to more than than $\eps |V(T^*)|$ backward edges in $\sigma$, where $T^*$ is the induced tournament on the interval spanning $I_1$ to $I_4$ and $\eps=1/(100t^2)$. We analogously remove the $O(1)$ vertices incident to many backward edges from $J_1 \cup J_4$ as well. By mimicking the proof of Claim~\ref{noktt}, we may now assume that there are no copies of $K_{t,t}$ in the backward edges between either $I_1$ and $I_4$ or $J_1$ and $J_4$

    Now, our plan is to find a $K_{r,t}$ using good backward edges, where the smaller partite class of size $r$ lies in $I_1$ and the larger partite class of $t$ vertices lies in $J_1$, and to similarly a copy of $K_{r,t}$ with the smaller partite class in $J_4$ and the larger partite class in $I_4$. This can obviously be done by Proposition~\ref{drc1} by what we know about the number of good backward edges between these sets. What is crucial however, is to find two such copies where all the other edges between the partite classes are forward edges directed from left to right; we will accomplish this with the help of Claim~\ref{noktt}.
    
    We argue as follows using dependant random choice. For suitably large constants $C_8 = C_8 (t) > 0$, $C_9 = C_9 (t) > 0$ and $C_{10} = C_{10} (t) > 0$ with $C_{10}$ sufficiently larger than $C_9$ and $C_9$ sufficiently larger than $C_8$, we may appeal to Proposition~\ref{drc1} to find two sets $A\subset I_1$ and $B\subset J_4$ such that
    \begin{enumerate}
    	\item $|A| = C_8,|B| = C_9$;
		\item the induced tournaments on $A$ and $B$ are transitive,
		\item all the edges between $A$ and $B$ are directed  from $A$ to $B$,
    	\item every subset of size $r$ in $A$ has at least $C_{10}$ common in-neighbours in $J_1$, and
    	\item every subset of size $r$ in $B$ has at least $C_{10}$ common out-neighbours in $I_4$.
    \end{enumerate}

	The second and third points above need some elaboration. Having found two suitably large sets in $I_1$ and $J_1$ meeting the other conditions, we appeal to Ramsey's theorem for tournaments to pass to a suitably large transitive subset within each. Subsequently, we know that there is no copy of a $K_{t,t}$ in the backward edges between these transitive subsets by Claim~\ref{noktt}, so we now appeal to Ramsey's theorem for bipartite graphs to find a suitably large complete bipartite graph between these transitive subsets with all the edges directed forwards from left to right.

	We repeat the above process of `two-step cleaning' via Ramsey's theorem across all pairs of the common in-neighbours of $r$-sets in $A$ and the common out-neighbours of the $r$-sets in $B$. In other words, we may now assume, that for a suitably large $C_{11} = C_{11} (t)>0$, we have the following: every $r$-set $A' \subset A$ has a set of $C_{11}$ common in-neighbours $A''$ in $J_1$ and every $r$-set $B' \subset B$ has a set of $C_{11}$ common out-neighbours $B''$ in $I_4$ such that the induced tournaments on $A''$ and $B''$ are transitive, and all the edges between $B''$ and $A''$ are directed forwards from $B''$ to $A''$. 
	
	Next, since $|B|$ is much larger than $|A|$, we refine $B$ as follows. For each $r$-set $A' \subset A$, we consider it's set $A''$ of $C_{11}$ common in-neighbours in $J_1$. We know by assumption that there is no copy of $K_{t,t}$ in the backward edges between $A''$ and $B$, so by Ramsey's theorem, we may find a sufficiently large complete bipartite graph directed forwards from $A''$ to $B$. We iterate through the $r$-sets in $A$ and use this observation to conclude (by passing to appropriate subsets) that for a suitably large constant $C_{12} = C_{12}(t) >0$, we have $|B| = C_{12}$ and every $r$-set $A' \subset A$ has a set of $t$ common in-neighbours $A''$ in $J_1$ such that all edges between $A''$ and $B$ are directed from $A''$ to $B$.
	
	We are now done since we may find a copy of $\DD_t$ as follows. We select any subset $B' \subset B$ of size $r$. Consider it's set of $B''$ of $C_{11}$ common out-neighbours in $I_4$. Since there is no copy of $K_{t,t}$ in the edges directed from $I_4$ and $I_1$, we may find, by Ramsey's theorem, a complete bipartite graph directed forwards between a set $X \subset B''$ of size $t$ and a set $A' \subset A$ of size $r$. Let $Y$ be the set of $t$ common in-neighbours of $A'$ in $J_1$. Then it is easy to see that the transitive classes $A' \cup B'$, $X$ and $Y$ together induce a copy of $\DD_t$. This completes the proof by contradiction.
\end{proof}
    
\section{Exceptional patterns}\label{sec:excep}
The results established earlier in the paper for unavoidable patterns of order at least three might suggest at first glance that the Ramsey numbers for patterns of order two should satisfy $\CC(2,\delta) = \Theta ((1/\delta)^2)$ and $\TT(2,\delta) = \Theta ((1/\delta))$. However, this is not the case; patterns of order two exhibit some degenerate behaviour, as we shall now demonstrate.

First, we deal with unavoidable $2$-colourings. While it is not hard to prove a much more precise result, we settle for the following.
\begin{proposition}
$\CC(2,\delta) = \Theta (1/\delta)$.
\end{proposition} 
\begin{proof}
By taking a colouring of $K_n$ where all edges are coloured blue except the edges incident with some vertex (which are coloured red), we obtain a colouring which does not contain a $K_4$ inducing an unavoidable $2$-colouring. Clearly, both colours appear on at least $n-1$ edges. 

Next, we shall show that there exists an absolute constant $C>0$ such that any colouring of $K_n$ where both colours appear on at least $Cn$ edges contains a $K_4$ inducing an unavoidable $2$-colouring. 

Suppose $G=K_n$ has a colouring where both colours appear at least $Cn$ times. Now, following the proof of Theorem~\ref{colourthm}, from Claim~\ref{claim:fewredvertices} (with $\varepsilon=1/6$), we are guaranteed that there is a set $S$ of at most $C_1$ vertices such that in $V(G)\setminus S$ every vertex is incident with at least $2n/3$ blue edges. Assuming $C>C_1$, we deduce $G\setminus S$ must span a red edge $xy$. Using the fact that the blue neighbourhoods of $x$ and $y$ must intersect in at least two vertices, we obtain a $K_4$ inducing an unavoidable $2$-colouring.
\end{proof}

The case of unavoidable $2$-tournaments is somewhat harder, and we are unfortunately unable to determine the correct rate of growth of $\TT(2,\delta)$. Nonetheless, we are able to show the following. 
\begin{proposition}
$ \log(1/\delta) / \delta \ll \TT(2,\delta) \ll (\log(1/\delta))^2 / \delta$.
\end{proposition}

\begin{proof}
First, we dispose of the lower bound using an inductive construction. Let $T_n$ be a tournament on $n$ vertices which does not contain a copy of $\DD_2$ and which is $\log n/(5n)$-far from being transitive; such a tournament exists when $n=3$, as can be seen from considering a cyclic triangle. Given $T_n$, we shall construct a tournament $T$ on $2n+1$ vertices with the required properties. To do so, we take two vertex-disjoint copies of $T_n$, say on vertex sets $A$ and $B$, and direct all the edges from $A$ to $B$. Then, we add a new vertex $z$ where all the edges are directed from $B$ to $z$ and from $z$ to $A$. We observe that this tournament does not contain a $\DD_2$. Indeed, any such copy must contain $z$ as $\DD_2$ is strongly-connected. Furthermore, note that $\DD_2$ contains two vertex-disjoint copies of a cyclic triangle. Therefore, one such copy must use $z$, and the other must be entirely inside $A$ or entirely inside $B$, but this is impossible. 

Now, we claim that $T$ is $\log(2n+1)/(5(2n+1))$-far from being transitive. To see this, observe that any ordering of $A$ must span at least $n\log n/5$ backward edges, and the same holds for $B$, by the induction hypothesis. Hence, any ordering of $V(T)$ must have at least $2n\log n/5$ backward edges from $E(T[A]) \cup E(T[B])$. Finally, note that one of the following must hold. Either there are $n/2$ vertices in $A$ which precede every vertex in $B$, in which case, regardless of where $z$ is in the ordering, $z$ must be incident to least $n/2$ backward edges, or the first vertex of $B$ in the ordering must be incident to at least $n/2$ backward edges from $A$. In either case, any ordering of $V(T)$ spans at least $n/2+2n\log n/5\geq (2n+1)\log(2n+1)/5$ backward edges, as claimed.

Next, we deal with the upper bound, again proceeding by induction on the number of vertices. The argument closely resembles the proof of Theorem~\ref{tourthm}, so we restrict ourselves to sketching the main points of departure. Clearly, it suffices to handle the case where the number of vertices $n$ is sufficiently large, say, greater than a sufficiently large constant $C > 0$. Let $T$ be a tournament on $n$ vertices which is $C(\log n)^2/n$-far from being transitive, and let $\sigma$ be an ordering of $V(T)$ which minimises the number of backward edges.

Let $I_1$ and $I_2$ be the intervals corresponding to the first half and the last half of $\sigma$ of sizes $n/2$ each. By the induction hypothesis, both $I_1$ and $I_2$ induce at most $C(n/2)(\log(n/2))^{2}$ backward edges in $\sigma$. Therefore, the number of backward edges 
from $I_2$ to $I_1$ is at least $    Cn(\log n)^2-Cn(\log(n/2))^{2}\geq Cn\log(n) / 2$. 

In the same fashion as in the proof of Theorem~\ref{tourthm}, let $J_1<J_2$ be two disjoint intervals for which the number of backward edges between $J_1$ and $J_2$ is $K(|J_1|+|J_2|)\log(|J_1|+|J_2|)$, with $K\geq C/2$ as large as possible. We proceed assuming $|J_2|\geq |J_1|$, the other case being analogous. Let $X\subset J_2$ be the initial segment of $J_2$ of size $|J_2|/10$. By the maximality of $K$, we know that the number of backward edges between $J_1$ and $X$ is at most $K(|J_1|+|X|)\log(|J_1|+|X|)\leq (19K/20)(|J_1|+|J_2|)\log(|J_1|+|J_2|)$. Therefore, the number of backward edges between $J_1$ and $J_3 =J_2\setminus X$ is at least $(K/20)(|J_1|+|J_3|)\log(|J_1|+|J_3|)$, and additionally, we also know that $d(J_1,J_3)\geq (|J_1|+|J_3|)/200$. Now, let $Y_1<Y_2$ be two disjoint intervals satisfying both the above properties where the number of backward edges between $Y_1$ and $Y_2$ is as at least $K'(|Y_1|+|Y_2|)\log(|Y_1|+|Y_2|)$, with $K'\geq K/20$ as large as possible.

A minor modification, losing a logarithmic factor, of the proof of Claim~\ref{densityincr} shows that we can split $Y_1=P_1\cup Q_1$ and $Y_2=P_2\cup Q_2$, with $P_1 < Q_1 < P_2 < Q_2$ such that the number of backward edges between $P_1$ and $P_2$ and between $Q_1$ and $Q_2$ are respectively at least $K_1(|P_1|+|P_2|)$ and $K_2(|Q_1|+|Q_2|)$. As before, we shall suppose that we have picked $P_1<Q_1<P_2<Q_2$ as above for which $K_1 + K_2$ is as large as possible.

We shall sketch how to handle the case where $|Q_2|\geq |P_2|$, the other case being analogous. Moreover, in the argument that follows, we may assume without loss of generality that $|Q_1|\leq |Q_2|$. 

Delete the first $|Q_2|/20$ vertices of $Q_2$ and denote the remaining interval by $Q'_2$. As in the proof of Theorem~\ref{tourthm}, this allows us to separate $Q'_2$ from $P_2$, since we now have $d(P_2,Q'_2) \geq (|P_2|+|Q'_2|)/50$. We now need to show that there are still sufficiently many backward edges between $Q_1$ and $Q'_2$. Indeed, $Q_1$ and $Q_2$ span $K_2(|Q_1|+|Q_2|)$ backward edges between them, and the maximality of $K_2$ (with $K_1$ fixed) allows us to bound from above the number of backward edges between $Q_1$ and $Q_2 \setminus Q'_2$, from which we may conclude that the number of backward edges between $Q_1$ and $Q'_2$ is at least $(K_2/200)(|Q_1|+|Q'_2|)$. Replacing $Q_2$ by $Q'_2$, we may now assume that $P_1<Q_1<P_2<Q_2$ are four intervals as above, again with the number of backwards edges between these intervals assumed to be maximal in the same sense as before. 

Now that we have separation between $Q_1$ and $P_2$ and between $P_2$ and $Q_2$, all that is left to do is to enforce some separation between $P_1$ and $Q_1$. We are led to handle two different cases, depending on whether or not $Q_1$ and $Q_2$ have comparable sizes. 

Suppose first that $|Q_1| \leq |Q_2| \le 4|Q_1|$. Then, we may proceed as we did before, deleting the first $|Q_1|/20$ vertices from $Q_1$ to create separation between $P_1$ and $Q_1$ while still ensuring that a positive fraction of the backward edges between $Q_1$ and $Q_2$ still survive; the rest of the argument is identical to the proof of Theorem~\ref{tourthm}. 

Next, suppose that $|Q_1| < |Q_2|/4$. Then, let $Q'_2$ be the initial segment of $Q_2$ of the same length as $Q_1$, and let $Q''_2=Q_2\setminus Q'_2$. If there are fewer than $(K_2/10)(|Q'_2|+|Q_1|)$ backward edges between $Q_1$ and $Q'_2$, then the intervals $Q_1$ and $Q''_2$ must have at least $K_2(|Q_1|+|Q_2|)-(K_2/10)(|Q'_2|+|Q_1|)>K_2(|Q''_2|+|Q_1|)$ backward edges between them, which is a contradiction. This implies that, replacing $Q_2$ by $Q'_2$ if necessary, we may assume $Q_1$ and $Q_2$ have comparable sizes. Now, we take four such intervals $P_1<Q_1<P_2<Q_2$ as above (with the appropriate separation between intervals), with the additional constraint that $Q_1$ and $Q_2$ have comparable sizes, and again  with the number of backwards edges between intervals assumed to be appropriately maximal; the rest of the argument is identical to the proof of Theorem~\ref{tourthm}. 
\end{proof}

\section{Conclusion}\label{conc}
Our main contribution in this paper was to pin down the order of magnitude of the Ramsey numbers $\CC(t,\delta)$ and $\TT(t, \delta)$ for fixed $t \in \N$ as $\delta \to 0$. If one is however willing to settle for just the correct exponents governing the growth rates of these Ramsey numbers, then more can be said. 

A careful rendering of our argument yields dependencies governed by iterated logarithms, allowing us to establish weaker forms of the bounds~\eqref{colbound} and~\eqref{tourbound} that are valid as long as $t \lll 1/\delta$. Concretely, we have
\[	\CC(t,\delta) = (1/\delta)^{t(1 + o(1))}\]
as $\delta \to 0$ with $3 \le t \le \log^{(3)} (1/\delta)$, and 
\[	\TT(t,\delta) = (1/\delta)^{\lceil t/2 \rceil (1+o(1))}\]
as $\delta \to 0$ with $3 \le t \le \log^{(4)} (1/\delta)$. It would be of interest to work out, even roughly, at what point the above bounds cease to be valid.

It would also be interesting to understand the growth rate of the Ramsey numbers $\CC(t,\delta)$ and $\TT(t, \delta)$ in the other off-diagonal regime where $t \ggg 1/\delta$, as well as in the diagonal regime where $t \approx 1/\delta$. Both these questions pose interesting challenges of their own, somewhat orthogonal to the problems under consideration here.

Finally, it is somewhat embarrassing that we are unable to pin down the rate of growth of $\TT(2,\delta)$. While we have managed to estimate this degenerate case here up to a logarithmic multiplicative factor, we suspect that our lower bound gives the correct rate of growth. It would be of interest to improve on our upper bound and demonstrate that $ \TT(2,\delta) \ll \log(1/\delta) / \delta$, thereby closing a small but annoying gap in the existing bounds.

\section*{Acknowledgements}
The first author wishes to acknowledge support by the EPSRC, grant no. EP/N019504/1. The second author wishes to acknowledge support from NSF grant DMS-1800521. We would like to thank Kamil Popielarz for several helpful remarks. 
\bibliographystyle{amsplain}
\bibliography{turan_patterns}

\providecommand{\bysame}{\leavevmode\hbox to3em{\hrulefill}\thinspace}
\providecommand{\MR}{\relax\ifhmode\unskip\space\fi MR }
\providecommand{\MRhref}[2]{%
  \href{http://www.ams.org/mathscinet-getitem?mr=#1}{#2}
}
\providecommand{\href}[2]{#2}
\begin{thebibliography}{10}

\bibitem{Conlon2009}
D.~Conlon, \emph{A new upper bound for diagonal {R}amsey numbers}, Ann. of
  Math. \textbf{170} (2009), 941--960.

\bibitem{cb}
J.~Cutler and B.~Mont\'{a}gh, \emph{Unavoidable subgraphs of colored graphs},
  Discrete Math. \textbf{308} (2008), 4396--4413.

\bibitem{upperramsey}
P.~Erd\H{o}s, \emph{Some remarks on the theory of graphs}, Bull. Amer. Math.
  Soc. \textbf{53} (1947), 292--294.

\bibitem{em}
P.~Erd\H{o}s and L.~Moser, \emph{On the representation of directed graphs as
  unions of orderings}, Magyar Tud. Akad. Mat. Kutat\'{o} Int. K\"{o}zl.
  \textbf{9} (1964), 125--132.

\bibitem{lowerramsey}
P.~Erd\H{o}s and G.~Szekeres, \emph{A combinatorial problem in geometry},
  Compositio Math. \textbf{2} (1935), 463--470.

\bibitem{fs}
J.~Fox and B.~Sudakov, \emph{Unavoidable patterns}, J. Combin. Theory Ser. A
  \textbf{115} (2008), 1561--1569.

\bibitem{drc}
\bysame, \emph{Dependent random choice}, Random Structures Algorithms
  \textbf{38} (2011), 68--99.

\bibitem{survey}
Z.~F\"{u}redi and M.~Simonovits, \emph{The history of degenerate (bipartite)
  extremal graph problems}, Erd\H{o}s centennial, Bolyai Soc. Math. Stud.,
  vol.~25, J\'{a}nos Bolyai Math. Soc., Budapest, 2013, pp.~169--264.

\bibitem{book}
R.~L. Graham, B.~L. Rothschild, and J.~H. Spencer, \emph{Ramsey theory},
  2\textsuperscript{nd} ed., Wiley-Interscience Series in Discrete Mathematics
  and Optimization, John Wiley \& Sons, Inc., New York, 1990.

\bibitem{kst}
T.~K\H{o}vari, V.~T. S\'{o}s, and P.~Tur\'{a}n, \emph{On a problem of {K}.
  {Z}arankiewicz}, Colloquium Math. \textbf{3} (1954), 50--57.

\bibitem{long}
E.~Long, \emph{Large unavoidable subtournaments}, Combin. Probab. Comput.
  \textbf{26} (2017), 68--77.

\bibitem{fpr}
F.~P. Ramsey, \emph{On a problem of formal logic}, Proc. London Math. Soc.
  \textbf{30} (1930), 264--286.

\bibitem{Thomason1988}
A.~Thomason, \emph{An upper bound for some {R}amsey numbers}, J. Graph Theory
  \textbf{12} (1988), 509--517.

\end{thebibliography}

\end{document}